\documentclass[11pt]{amsart}  
\usepackage{amssymb, latexsym, amsmath, amsfonts,cases}

\makeatletter
\renewcommand*\env@matrix[1][*\c@MaxMatrixCols c]{%
  \hskip -\arraycolsep
  \let\@ifnextchar\new@ifnextchar
  \array{#1}}
\makeatother

\newtheorem{thm}{Theorem}[section]
\newtheorem{cor}[thm]{Corollary}

\newtheorem{prop}[thm]{Proposition}
\theoremstyle{definition}

\numberwithin{equation}{section}

\newcommand{\abs}[1]{\left\vert#1\right\vert}

\newcommand{\mbb}{\mathbb}

\begin{document}
\title{Analyzing the Wu metric on a class of eggs in $\mathbb{C}^n$ -- II}
\keywords{Wu metric, Kobayashi metric, negative holomorphic curvature}
\thanks{Both the authors were supported by the DST-INSPIRE Fellowship of the Government of India.}
\subjclass{Primary: 32F45; Secondary: 32Q45, 32H15}
\author{G. P. Balakumar and Prachi Mahajan}

\address{G. P. Balakumar:  Indian Statistical Institute, Chennai -- 600113, India.}
\email{gpbalakumar@isichennai.res.in}

\address{Prachi Mahajan: Department of Mathematics, Indian Institute of Technology -- Bombay, Mumbai 400076, India.}
\email{prachi.mjn@iitb.ac.in}

\pagestyle{plain}
\begin{abstract}
We study the Wu metric for the non-convex domains of the form
\[
E_{2m} = \big\{ z \in \mbb C^n : \abs{z_1}^{2m} + \abs{z_2}^2 + \ldots + \abs{z_{n-1}}^2 + \abs{z_n}^{2} <1  \big\},
\]
where $ 0 < m < 1/2$. Explicit expressions for the Kobayashi metric and the Wu metric on such pseudo-eggs $E_{2m}$ are obtained. The Wu metric is
then verified to be a continuous Hermitian metric on $ E_{2m} $ which is real analytic everywhere except 
along the complex hypersurface $ Z = \{ (0, z_2, \ldots, z_n ) \in E_{2m} \} $. We also show that the holomorphic sectional curvature of the Wu metric for this 
non-compact family of pseudoconvex domains is bounded above in the sense of currents by a negative constant independent of $m$. This
verifies a conjecture of S. Kobayashi and H. Wu for such $E_{2m}$.
\end{abstract}

\maketitle
\section{Introduction}
\noindent We continue our study of the Wu metric from \cite{BM} by focusing
on the following class of non-convex pseudo-egg domains
\begin{equation} \label{E0}
E_{2m} = \big\{ z \in \mbb C^n : \abs{z_1}^{2m} + \abs{z_2}^2 + \ldots + \abs{z_{n-1}}^2 + \abs{z_n}^{2} <1  \big\},
\end{equation}
for $ 0 < m < 1/2 $. Such a pseudo-egg cannot be biholomorphically transformed to any bounded convex domain. This follows 
by comparing the Kobayashi indicatrices, which must be linearly equivalent if the domains are biholomorphic to each other. Indeed, 
the linear mapping on the tangent space given by the derivative of the biholomorphism renders an equivalence between the Kobayashi
indicatrices at the corresponding points. Linear maps preserve convexity and Kobayashi indicatrix of a bounded convex domain is
convex. On the other hand, the indicatrix  at the origin for a pseudo-egg, (being a copy of $E_{2m}$), is non-convex. This also means that 
the Kobayashi metric fails to satisfy the triangle inequality on the tangent space at the origin $ T_0 E_{2m} $. While the 
Wu metric is indeed a norm, it is not clear if the Wu metric enjoys better regularity than the Kobayashi metric. In general, the Wu
metric may fail to be upper semicontinuous \cite{JP} notwithstanding the fact that the Kobayashi metric is always upper 
semicontinuous. In fact, in the case of $C^2$-smooth convex eggs, i.e. when $m>1$, the Wu metric is only $C^1$-smooth  (see \cite{BM}) while 
the Kobayashi metric is $C^2$-smooth. For the non-convex pseudo-eggs as in (\ref{E0}), first note 
that $\partial E_{2m}$ is not even $C^1$-smooth. Specifically, the non-smooth points of the boundary are given by
$\{ (z_1, \ldots, z_n) \in \partial E_{2m} :  z_1=0 \}$, which is also the boundary of the set $ Z = \{ (0, z_2, \ldots, z_n ) \in E_{2m} \} $; remaining piece of the boundary 
$\partial E_{2m} \setminus Z$ is a smooth strongly pseudoconvex hypersurface. Let us now state our result on the Wu metric for such pseudo eggs.

\begin{thm} \label{T0}
For $ 0 < m < 1/2 $, the Wu metric on $ E_{2m} $ is a continuous Hermitian metric which is real analytic on $ E_{2m} $ except along 
the thin set $Z$. It is nowhere-K\"ahler. Furthermore, its holomorphic sectional curvature is non-constant 
and is bounded above by a negative constant independent of $m$, in the sense of currents. 
\end{thm}

\noindent This result extends the work of Cheung and Kim (\cite{CK1} and \cite{CK2}) on pseudo-egg domains in $ \mathbb{C}^2 $. It also verifies the 
following conjecture of Kobayashi (\cite{Ko}) (refer \cite{Wu} 
for a modified version of this conjecture due to Wu) for the pseudo-eggs $E_{2m} $.  
\medskip\\
\noindent \underline{Conjecture K-W:} On every Kobayashi complete hyperbolic complex manifold, there exists a $C^k$-smooth (for some $k \geq 0$) complete Hermitian metric with its 
holomorphic curvature \footnote{The term `holomorphic
curvature' stands precisely for the holomorphic {\it sectional} curvature and is said to be strongly negative, if it is bounded above by a negative constant.} bounded above by a negative constant in the sense of currents.
\medskip\\
\noindent {\it Acknowledgements:} We thank our advisor Kaushal Verma for suggesting us this problem.

\section{The Wu metric on $ E_{2m} $ for $ 0 < m < 1/2 $}

\noindent  In order to analyse the Wu metric on $ E_{2m} $, it is natural to first compute the Kobayashi metric on $ E_{2m} $. 
Note that $ E_{2m} $ is a \textit{balanced} pseudoconvex domain in $ \mathbb{C}^n $ and hence the Kobayashi metric at the origin, 
$ K_{E_{2m}} \big( (0, 0, \ldots, 0) ,v \big) = q_{E_{2m}}(v)$ where $q_{E_{2m}}$ denotes the \textit{Minkowski functional} of $E_{2m}$. Let
$ \langle \cdot, \cdot \rangle $ denote the standard Hermitian inner product in $ \mathbb{C}^{n-1} $, and write 
$ z \in \mathbb{C}^n $ as $ z= (z_1, \hat{z})$, where $ \hat{z} = (z_2, \ldots, z_n) $. For $ (p_1, \ldots, p_n) \in E_{2m} $ with $ p_1 \neq 0 $
and $ \Psi $ an automorphism of $ \mathbb{B}^{n-1} $ that takes $ \hat{p} $ to the origin,
\begin{equation} \label{E1}
\Phi (z_1, \ldots, z_n) = \left( \frac{|p_1|}{p_1} \frac{\left( 1 - \abs{\hat{p}}^2 \right)^{1/2m} } { \left( 1 - \langle \hat{z}, 
\hat{p} \rangle \right)^{1/m} }  z_1, \Psi( \hat{z} ) \right)
\end{equation}
is an automorphism of $ E_{2m} $. Furthermore, 
\[
 \Phi \big( (p_1, \ldots, p_n )  \big) = \left( \frac{ |p_1| } {(1 - |p_1|^2)^{1/{2m} } } , \hat{0} \right).
 \]
It follows that the study of the Kobayashi and Wu metrics can be focussed on points $ (p, \hat{0}) $ for $ 0 \leq p < 1 $, as in the case of convex eggs, in 
view of automorphisms of $ E_{2m} $ of the form (\ref{E1}). 

\begin{thm} \label{T4}
 For $ 0 < m < 1/2 $, the Kobayashi metric for $E_{2m}$ at the point $(p, \hat{0})$ for $0<p<1$ is given by
\begin{alignat*}{4}
K_{E_{2m}} \big( (p, \hat{0}), v \big) = 
   \left\{ \begin{array}{lrl}
K_1 \big( (p, \hat{0}), v \big) = \frac{m \alpha (1-t) |v_1|}{p(1- {\alpha}^2) \left(m(1-t) + t \right) }
& \mbox{for} &  w  \leq 1, \\
\\
K_2 \big( (p, \hat{0}), v \big) = \left( \frac{m^2 p^{2m-2} |v_1|^2}{(1-p^{2m})^2}
+ \frac{|v_2|^2}{1- p^{2m}} + \cdots + \frac{|v_n|^2}{1- p^{2m}} \right)^{1/2} & \mbox{for}  & w  \geq \frac{1}{4m(1-m)}, \\
\\
\min \{ K_1 \big( (p, \hat{0}), v \big), K_2 \big( (p, \hat{0}), v \big) \}  \qquad \mbox{for} \; 1 <  w < \frac{1}{4m(1-m)},
\end{array}
\right.
\end{alignat*}
where $ v = (v_1, \ldots, v_n ) $ is a tangent vector at the point $ (p, \hat{0}) $, 
\begin{eqnarray}
w & = & \frac{ p^2 \left( |v_2|^2 + \cdots + |v_n|^2 \right) }{m^2 |v_1|^2}, \label{E32} \\
t & = & \frac{2m^2 w}{1 + 2m(m-1) w +  \big( 1 + 4m(m-1) w \big)^{1/2}}, \label{E33}
\end{eqnarray}
and $ \alpha $ is the unique positive solution of the equation $  \alpha^{2m} - t \alpha^{2m-2} - (1-t) p^{2m} = 0 $ in 
the interval $ (0,1) $. Moreover, there is a $ 1 < w_0 < \frac{1}{4m(1-m)} $ such that
\begin{alignat*}{3}
 K_{E_{2m}} \big( (p, \hat{0}), v \big) &= K_1 \big( (p, \hat{0}), v \big) \qquad \mbox{for} \; \; w \leq w_0,\\
 K_{E_{2m}} \big( (p, \hat{0}), v \big) &  = K_2 \big( (p, \hat{0}), v \big) \qquad \mbox{for} \; \; w \geq w_0.
\end{alignat*}
Furthermore, $ w_0 = \frac{t_0}{ \left(m + (1-m) t_0 \right)^2 }  $ where $ t_0 = \frac{x_0^{2m} - p^{2m} }{ x_0^{2m - 2} - p^{2m}} $
and $ x_0 $ is the solution of the equation
\begin{multline*}
 - (1-m)^2 x^{4m} + \left( -1 - 2m + 2 m^2 + p^{2m} \right) x^{4m-2} - m^2 x^{4m -4} + \left( 1 - (2m - 1) p^{2m} \right) x^{2m} \\
+ \left( 1 + (2m -1) p^{2m} \right) x^{2m-2} - p^{2m} = 0.
\end{multline*}
\end{thm} 

\noindent The above result extends the computation done for such pseudo-egg domains in $ \mathbb{C}^2 $ (Theorem 4 of \cite{PZ-1996}). The proof 
relies on the description of complex geodesics with respect to the Kobayashi metric due to Pflug 
and Zwonek (see Proposition 2 of \cite{PZ-1996}). The proof of Theorem \ref{T4} follows using the techniques of \cite{PZ-1996}.

\medskip

\begin{prop} \label{P6}
In terms of the Euclidean coordinates on the tangent bundle $ E_{2m} \times \mathbb{C}^n $, for every $ (p, \hat{0}) \in E_{2m} $, the unit 
sphere of the Wu metric in $ T_{(p, \hat{0})} E_{2m} $ is given by
\[
 r_1 |v_1|^2 + r_2 \left( |v_2|^2 + \ldots + |v_n |^2 \right) = 1
\]
where $ r_1 $ and $ r_2 $ are positive real-valued continuous functions of $ p $. 
\end{prop}

\medskip

\noindent It turns out that, to determine the best fitting ellipsoid at $  (p, \hat{0}) $, it suffices to find $ r_1, r_2 > 0 $ such that 
the set
\[
\left\{ (v_1, \ldots, v_n) : v_1 \geq 0, \ldots, v_n \geq 0, 
r_1 v_1^2 + r_2 \left( v_2^2 + \ldots + v_n^2 \right) = 1 \right\} 
\]
encloses the smallest volume with the coordinate axes and contains the set 
\begin{eqnarray} \label{E41}
\left\{ (v_1, \ldots, v_n) : v_1 \geq 0, \ldots, v_n \geq 0, 
K_{E_{2m}} \left( ( p, \hat{0} ), v \right) \leq 1 \right\}.
\end{eqnarray}
Recall from Theorem \ref{T4} that there is a $ 1 < w_0 < \frac{1}{4m(1-m)} $ such that
\begin{alignat*}{4}
K_{E_{2m}} \big( (p, \hat{0}), v \big) = 
   \left\{ \begin{array}{lrl}
K_1 \big( (p, \hat{0}), v \big) & \mbox{for} &  w  \leq w_0, 
\\
K_2 \big( (p, \hat{0}), v \big)   & \mbox{for} & w \geq w_0.
\end{array}
\right.
\end{alignat*}
It follows that the boundary of the set described by equation (\ref{E41}) is the union of two curves, determined by whether $ w < w_0 $ or $ w \geq w_0 $.
Henceforth, the portion of the curve $ K_{E_{2m}} ^2 \left( ( p, \hat{0} ), v \right) = 1 $ for $ w \geq w_0 $ and $ w < w_0 $ will be referred to as 
the lower K-curve and the upper K-curve respectively. In terms of the square transformation
\begin{eqnarray*}
 x & = &  v_2^2 + \ldots + v_n^2 \; \mbox{and}\\
 y & = &  v_1^2,
\end{eqnarray*}
the lower $ K$-curve is described by
\begin{equation*} 
\frac{m^2 p^{2m-2} \; y }{(1-p^{2m})^2} + \frac{x}{1- p^{2m}} = 1.
\end{equation*}
The upper $K$-curve is given by the parametric equations, as follows.
%\begin{alignat*}{3} 
\begin{subnumcases}{}
\label{E42}
x(\alpha)  =  \big( v_2 (\alpha)\big)^2 + \ldots +  \big( v_n (\alpha) \big) ^2 = \frac{ \alpha^{4m-2} + p^{4m} - p^{2m} \alpha^{2m-2} 
- p^{2m} \alpha^{2m} } {\alpha^{4m-2}}, \; \mbox{and} \\ \label{E43}
y(\alpha)  =  \big( v_1 (\alpha) \big) ^2 =  \left(\frac{ p \big( m \alpha^{2m-2} - (m-1) \alpha^{2m} - p^{2m} \big)}
{m \alpha^{2m-1} } \right)^2, 
%\end{alignat*}
\end{subnumcases}
with $\alpha$ varying between $p$ and $x_0$, where $ x_0 $ is as in Theorem \ref{T4}. As the upper $K$-curve can be represented as the graph of a function,  we may cast the equations (\ref{E42})
and (\ref{E43}) as a single equation 
 $ y( \alpha) = \left( f \big( \sqrt{x(\alpha)} \big) \right)^2 $. The parametric form for the upper $K$-curve as above, shows that $f$ is real analytic.
It is straightforward to verify that $f$ is strictly square convex for $ p < \alpha < x_0 $ (see \cite{CK1}). This fact together with Proposition \ref{P6} 
renders the following result.
\begin{prop} \label{T6} 
The Wu metric on $ E_{2m} $ for $ m < 1/2 $ at the point $ (p, \hat{0})$, $ 0 < p <1 $ is given by
\begin{equation*}
h_{E_{2m}} (p, \hat{0}) = \frac{1} {\left( 1- p^2 \right)^2} d z_1 \otimes d \overline{z}_1 + \frac{1} {\left( 1 - p^{2m} \right) } 
d z_2 \otimes d \overline{z}_2 + \ldots + \frac{1} {\left( 1 - p^{2m} \right) } d z_n \otimes d \overline{z}_n. 
\end{equation*}
\end{prop}

\noindent Further, using the invariance of the Wu metric under automorphisms of $ E_{2m} $ (as described in (\ref{E1})), we arrive at the following result.

\begin{thm} \label{T5}
For $ 0< m < 1/2$, the Wu metric on $ E_{2m} $ at the point $ (z_1, \ldots, z_n) $ is given by
\begin{equation*}
 \sum_{i,j=1}^n h_{i\bar{j}} (z_1, \ldots, z_n) \; d z_i \otimes d \overline{z}_j,
\end{equation*}
where 
\begin{alignat*}{3}
 h_{1 \bar{1}} ( z_1, \ldots, z_n) & = \frac{ \left(1 - |\hat{z}|^2 \right)^{1/m} } {  \left( \left(1 - |\hat{z}|^2 \right)^{1/m} - |z_1|^2 \right)^2 }, \\
 h_{1 \bar{j}} ( z_1, \ldots, z_n) & =  \frac{  \left(1 - |\hat{z}|^2 \right)^{-1 + 1/m} \overline{z}_1 z_j} 
{ m  \left( \left(1 - |\hat{z}|^2 \right)^{1/m} - |z_1|^2 \right)^2 } \; \; \mbox{for} \; \; 2 \leq j \leq n, \\
h_{i\bar{1}} (z_1, \ldots, z_n) & = \overline{ h_{1 \bar{i} } } (z_1, z_2, \ldots, z_n) \; \; \mbox{for} \; \; 2 \leq i \leq n, \\
h_{j\bar{j}} (z_1, \ldots, z_n) & = \left( \frac{  \left(1 - |\hat{z}|^2 \right)^{-2 + 1/m} |z_1|^2 |z_j|^2 } 
{ m^2  \left( \left(1 - |\hat{z}|^2 \right)^{1/m} - |z_1|^2 \right)^2 } + \frac{1 - |\hat{z}|^2 + |z_j|^2 }{ \left( 1 - | \hat{z}|^2 \right) 
\left( 1 - |\hat{z}|^2 - |z_1|^{2m} \right) } \right) \; \mbox{for} \; 2 \leq j \leq n, \\
h_{i\bar{j}} (z_1, \ldots, z_n) & = \left( \frac{  \left(1 - |\hat{z}|^2 \right)^{-2 + 1/m} |z_1|^2 \overline{z}_i z_j } 
{ m^2  \left( \left(1 - |\hat{z}|^2 \right)^{1/m} - |z_1|^2 \right)^2 } + \frac{ \overline{z}_i z_j }{ \left( 1 - | \hat{z}|^2 \right) 
\left( 1 - |\hat{z}|^2 - |z_1|^{2m} \right) } \right) \; \mbox{for} \; 2 \leq i, j \leq n
 \end{alignat*}
 and $ i \neq j $. 
\end{thm}

\noindent It follows that the Wu metric is real analytic  on the set $  \{ (z_1, \ldots, z_n ) \in E_{2m} : z_1 \neq 0 \} $.
Moreover, the Wu metric is continuous at points $ (0, z_2, \ldots, z_n ) $ of $ E_{2m} $. Indeed, the continuity and 
completeness of the Wu metric on $ E_{2m} $ follows from Proposition $4$ of \cite{JP}. Completeness of the Wu metric here 
relies on Kobayashi completeness of the domains $E_{2m}$, which is guaranteed  by \cite{Pflug-1984}. Furthermore, it is straightforward to see that
\[
 \frac{ \partial h_{12} }{ \partial z_2} (z_1, \hat{0}) \neq \frac{ \partial h_{22} }{ \partial z_1} (z_1, \hat{0}). 
\]
This shows (see, for instance, Lemma 3.2 of \cite{GF}) that the Wu metric is not K\"ahler at these reference points. Since 
$ E_{2m} \setminus Z$ is the orbit of $ \{(z_1, \hat{0}) \in E_{2m} : z_1 \neq 0 \}$ under the action of the automorphism group of $ E_{2m} $, it follows 
that the Wu metric is not K\"ahler on $ E_{2m} \setminus Z$. Notice that $ E_{2m} \setminus Z $ is an open dense subset of $ E_{2m} $ and hence, the
following result.
\begin{cor}
For any $0<m <1/2$, the Wu metric on $ E_{2m} $ is nowhere K\"{a}hler. 
\end{cor}
\section{Negative holomorphic sectional curvature}

\noindent Let $ G = \sum_{i,j=1}^n g_{i \bar{j} } dz_i \otimes d \overline{z}_j $ be a $C^2$-smooth Hermitian metric on a complex manifold $ X $. Then
the holomorphic sectional curvature of $ G $ along the direction of $ \xi = ( \xi_1, \ldots, \xi_n) \in T_p X $ at $ p \in X $ is given by
\[
 \frac{ \displaystyle\sum R_{i \bar{j} k \overline{l} } (p) \xi_i \overline{\xi}_j \xi_k \overline{\xi}_l } { 
  \displaystyle\sum g_{i \bar{j} } (p)  g_{k \bar{l} } (p) \xi_i \overline{\xi}_j \xi_k \overline{\xi}_l  }, 
\]
where $ R_{i \bar{j} k \overline{l} } $ are the components of the 
curvature tensor given by
\[
 R_{i \bar{j} k \overline{l} } = - \frac{ \partial ^2 g_{i \bar{j} } }
 { \partial z_k \partial \overline{z}_l } + \sum_{ \alpha, \beta} g^{ \alpha \bar{\beta} } \frac{ \partial g_{i \bar{\beta}} }{ \partial z_k }
  \frac{ \partial g_{\alpha \bar{j}} }{ \partial \overline{z}_l}.
\]
Here, $ ( g^{\alpha \bar{\beta} } ) $ denotes the inverse of the matrix  $ ( g_{\alpha \bar{\beta} } ) $.
 
\medskip

\noindent In case $ G $ is only continuous (and not $C^2$-smooth), then the holomorphic sectional curvature 
is defined in a distributional sense, as a \textit{current} of type $(1,1) $ (cf. following \cite{Wu}) and is said to be
bounded above by a negative constant $ c $ if every embedded Riemann surface $ S $ in $ X $ with $ H \big|_S = h_0 \; d \xi \otimes d \bar{\xi} $
satisfies
\[
 \Delta_{\xi} \log h_0 = \frac{ \partial ^2 \log h_0}{ \partial \xi \partial \bar{\xi} } > -c h_0 \; \partial \xi \wedge \partial \bar{\xi},
\]
in the sense of currents.

\begin{prop}
The holomorphic sectional curvature of the Wu metric on $ E_{2m} $, $ 0 < m < 1/2 $, is bounded above by $ -1/2 $ at every point where the Wu metric is smooth. 
\end{prop}

\begin{proof}
A direct computation using Theorem \ref{T5} shows that for $ 0 < p < 1 $,
\begin{alignat*}{3}
R_{1 \bar{1} 1 \bar{1} } \big( (p, \hat{0} ) \big) & =  - \frac{2}{ (1- p^2)^4}, \\
R_{1 \bar{1} j \bar{j} } \big( (p, \hat{0} ) \big) & = - \frac{1 + p^2} {m ( 1-p^2)^3} + \frac{p^2 ( 1- p^{2m}) }{m^2 ( 1-p^2)^4} \; \mbox{for} \; 2 \leq j \leq n,\\
R_{1 \bar{i} i \bar{1} } \big( (p, \hat{0} ) \big) = R_{i \bar{1} 1 \bar{i} } \big( (p, \hat{0} ) \big) & = - \frac{1 + p^2} {m ( 1-p^2)^3} + \frac{p^{2m}}
{(1-p^2)^2 (1-p^{2m}) } \; \mbox{for} \; 2 \leq i \leq n,\\
R_{i \bar{i} 1 \bar{1} } \big( (p, \hat{0} ) \big) & = - \frac{m^2 p^{2m-2}} { (1-p^{2m})^3 } \; \mbox{for} \; 2 \leq i \leq n,\\
R_{i \bar{i} j \bar{j} } \big( (p, \hat{0} ) \big) & = - \frac{1} { (1-p^{2m})^2 } \; \mbox{for} \; 2 \leq i, j \leq n \; \mbox{and} \; i \neq j,\\
R_{i \bar{j} j \bar{i} } \big( (p, \hat{0} ) \big) & = - \frac{p^2}{m^2(1-p^2)^2} - \frac{1}{1-p^{2m}} \; \mbox{for} \; 2 \leq i, j \leq n \; \mbox{and} \; i \neq j,\\
R_{j \bar{j} j \bar{j} } \big( (p, \hat{0} ) \big) & = - \frac{p^2}{m^2 ( 1-p^2)^2 } - \frac{1}{ 1 - p^{2m} } - \frac{1}{ (1 - p^{2m})^2 } \; \mbox{for} \; 
2 \leq j \leq n,
\end{alignat*}
and all other curvature components vanish. Now, arguments similar to those employed in \cite{CK1} using the above computations yield that
\[
 \sum R_{i \bar{j} k \overline{l} } (p) \xi_i \overline{\xi}_j \xi_k \overline{\xi}_l < - \frac{1}{2} \sum g_{i \bar{j} } (p)  g_{k \bar{l} } (p) 
 \xi_i \overline{\xi}_j \xi_k \overline{\xi}_l. 
\]
\end{proof}
\noindent It remains to establish the uniform negativity of curvature on the thin set $Z$. 
%whose boundary meets $\partial E_{2m}$ along its cuspidal singular subset namely, $\{ z \in \partial E_{2m} \; : \;  z_1=0 \}$.
\begin{prop}
There is a negative constant $c$ such that the holomorphic sectional curvature of the Wu metric 
at points of $ Z \subset E_{2m} $, is bounded above by $c$ in the sense of currents. Moreover, the constant $c$ is independent of $m$.
\end{prop}
\noindent To prove this proposition, one cannot rely solely on the arguments used in the case of convex eggs (i.e., those presented 
in Appendix B of \cite{CK1}), since the Wu metric is not even $C^1$-smooth on $ Z $. However, the pseudo-eggs have the following property -- for $ m < 1/2$, $ E_{2m} \subset \mathbb{B}^n $, 
where $ \mathbb{B}^n $ denotes the unit ball in $ \mathbb{C}^n $. Hence, one can pull back the standard Poincare-Bergman metric $g$ on $ \mathbb{B}^n $ by the 
inclusion mapping $i:E_{2m} \to \mathbb{B}^n$, to get the metric $i^{*}g$ on $ E_{2m} $ and compare it with the Wu metric $h_{E_{2m}}$. It 
follows that $i^{*} g \leq \sqrt{n}h_{E_{2m}}$. Moreover, in this case, the best fitting ellipsoid at the origin is $\mathbb{B}^n$. 
Furthermore, since the Kobayashi-indicatrix at the origin is (a copy of) $E_{2m}$, it follows that $i^{*}g$ coincides with 
$h_{E_{2m}} $ at the origin.  Now, let $S$ be any embedded Riemann surface with complex coordinate $s$ and passing through the origin at 
$s=0$. To establish the strong negativity of the holomorphic curvature near the origin, in the sense of currents, compare the 
restriction $ G =i^{*}g_{\vert_S} $ with $H=h_{E_{2m}}{\vert_S}$. Real analyticity of $G$ ensures that its 
logarithmic average on small `discs' about $s=0$ in $S$ equals $\Delta \log G(s)$ at $s=0$ -- this can be verified by a Taylor
expansion of $G$ in $s, \overline{s}$. The decreasing property of holomorphic curvature together with the facts that $G$ is a
conformal metric on $S$ of constant negative holomorphic curvature and  
$H \geq 1/\sqrt{n}G$, establishes the strong negativity of the holomorphic curvature current of the Wu metric in a neighbourhood
of the origin as in \cite{CK2}. Notice that the bound on the curvature that we get here, is a constant that depends on the dimension $n$.
But this constant does not depend on $m$. As $Z$ is the orbit of the origin under the action of ${\rm Aut}(E_{2m})$, this analysis carries 
forth to hold throughout $Z$.

\end{document}